\documentclass[twoside,12pt,a4paper]{amsart}

\usepackage{amssymb}
\usepackage{amsrefs}
\usepackage{ mathrsfs }

\newcommand{\cl}{C\kern -0.1em\ell}

\newtheorem{thm}{Theorem}[section]
\newtheorem{corollary}[thm]{Corollary}
\newtheorem{lemma}[thm]{Lemma}
\newtheorem{proposition}[thm]{Proposition}
\theoremstyle{definition}
\newtheorem{definition}[thm]{Definition}

\newtheorem{remark}[thm]{Remark}

\newtheorem{notation}[thm]{Notation}

\numberwithin{equation}{section}

\DeclareMathOperator{\tr}{tr}

\providecommand{\href}[2]{#2}

\newcommand{\BA}{\mathbb{A}}
\newcommand{\BC}{\mathbb{C}}
\newcommand{\BD}{\mathbb{D}}
\newcommand{\BK}{\mathbb{K}}
\newcommand{\BP}{\mathbb{P}}	
\newcommand{\BR}{\mathbb{R}}
\newcommand{\fpower}[1]{{}^2 \kern -0.0125em #1} %left power
\newcommand{\Mob}[1]{\mathrm{M\ddot{o}b}(#1)}
\newcommand{\GL}[1]{\mathrm{GL}_2(#1)}
\newcommand{\SL}[1]{\mathrm{SL}_2(#1)}
\newcommand{\Mat}[1]{\mathrm{M}_2(#1)}

\DeclareMathOperator{\Imm}{\mathrm{Im}}

\newcommand{\ed}{\end{document}}
	
\begin{document}
%-------------------------------------------------------------------------
% editorial commands: to be inserted by the editorial office
%
%\firstpage{1} \volume{228} \Copyrightyear{2004} \DOI{003-0001}
%\seriesextraline{This is the Concrete Title of this Book\br H.E. R and S.T.C. W, Eds.}
%\firstpage{1}
\title[The Groups of Two by Two Matrices in Double and Dual Numbers]{The Groups of Two by Two Matrices in\\
Double and Dual Numbers, and\\ Associated  M\"{o}bius Transformations }

\author{Khawlah A. Mustafa}
\address{School of Mathematics \\
University of Leeds\\
UK.}
\email{mmkam@leeds.ac.uk}

\address{School of Science \\
University of Kirkuk\\
Iraq.}
\email{alikhawlah@yahoo.com}



\date{September 11, 2018}

\begin{abstract}
M\"{o}bius transformations have been studied over the field of complex numbers. In this paper, we investigate M\"{o}bius transformations over two rings which are not fields: the ring of double numbers and the ring of dual numbers. We give types of continuous one-parameter subgroups of $\GL{\fpower{\BR}},$ $\SL{\fpower{\BR}},$ $\GL{\BD},$ and 
$\SL{\BD}.$

\end{abstract}
%%%%%%%%%%%%%%

\maketitle
%%%%%%%%%%%%%%
\section{Introduction}
M\"{o}bius transformations, invented in 19th century  \cite{MR0462890}, have been extensively studied over the field of complex and real numbers, see \citelist{ \cite{MR2743058}*{Ch. 9}} for a comprehensive presentation. The purpose of this paper is to expand these ideas to double and dual numbers. Some new and unexpected phenomena will appear in those cases. Relying on the four types of continuous one-parameter subgroups of $\SL{\BR},$ we build, up to similarity and rescaling, all different types of continuous one-parameter subgroups of $\GL{\fpower{\BR}},$ $\SL{\fpower{\BR}},$ $\GL{\BD}$ and $\SL{\BD}$ where $\fpower{\BR}=\BR \oplus \BR.$  The rest of the introduction gives a detailed over\-view of this work. In the first subsection, we start from a review of M\"{o}bius transformations over the real and complex field. The second subsection explains the link between Clifford algebras and generalisation of M\"{o}bius transformations. The third subsection introduces the projective lines and M\"{o}bius transformations over two rings, the ring of double and dual numbers. The last subsection presents our main results and the outline of the paper.
%%%%%%%%%%%%%%%
%%%%%%%%%%%%%%%
\subsection{Real and Complex Projective Lines and M\"{o}bius Transformations}

Our results for dual and double numbers will be compared with the known constructions in 
$\BR$ and $\BC.$ For a reader's convenience, we briefly recall main points in a suitable form, further particularities can be found in \cite{MR2743058}. Let $\BK$ be a field of real or complex numbers. Let ${\sim}$ be an equivalence relation on $\BK^2{\setminus}\{(0,0)\}$ defined as follows: $(z_1,z_2) \sim (z_3,z_4)$ if and only if there exists a non-zero number $u\in\BK$ such that $z_1=uz_3$ and $z_2=uz_4.$ The set of all equivalence classes $\BK^2/\sim$ is called the \textbf{projective line} over $\BK$. It is denoted by 
$\BK^2/\sim$ or by $\BP(\BK),$ for short. The point of the projective line corresponding to a vector $\begin{pmatrix}x\\y \end{pmatrix}$ is denoted by $[x:y].$ There is a natural embedding $x\mapsto[x:1]$ of the field $\BK$ into the projective line. The only point, $[1:0],$ not covered by this embedding is associated with infinity (ideal element) \citelist{\cite{Beardon01}\cite{MR2977041}*{Ch. 8}\cite{Kisil15a}\cite{Ulrych}}.

A linear transformation of $\BK^2$ can be represented by a $2\times2$-matrix 
$\begin{pmatrix}a&b\\c&d \end{pmatrix}$ multiplying two-dimensional vectors. The transformation is not degenerate (i.e., it is invertible) when $ad-bc\neq0.$ The collection of all $2\times 2$-matrices, 
$A=\begin{pmatrix}a&b\\c&d\end{pmatrix}$, such that $ad-bc\neq0,$ is a group denoted by $\GL{\BK}.$ The collection of all $2\times 2$-matrices $A=\begin{pmatrix}a&b\\c&d\end{pmatrix}$ such that $ad-bc=1$ is denoted by $\SL{\BK}$. It is a subgroup of $\GL{\BK}.$  A linear map $\BK^2\rightarrow \BK^2$ is a class invariant for 
$\sim.$ Thus, the linear transformation of $\BK^2$ produces the map $\BP(\BK)\rightarrow\BP(\BK)$ as follows:
$$
\begin{pmatrix}
a&b\\c&d
\end{pmatrix}[x:y]=[ax+by:cx+dy],\quad ad-bc\neq0,
$$
where $a,b,c,d\in\BK.$ For any $u\neq0,$ $\begin{pmatrix}
                                                a&b\\c&d
                                           \end{pmatrix}$ and 
																					$\begin{pmatrix}
                                                ua&ub\\uc&ud
                                           \end{pmatrix}$ 
define the same map of $\BP(\BK).$ In other words, for all 
$A=\begin{pmatrix}
          a&b\\c&d
   \end{pmatrix}\in \GL{\BK}$ such that $ad-bc>0$ for $\BK=\BR$ and $ad-bc\neq0$ for 
	$\BK=\BC,$ there is 
$$
A^\prime=\frac{1}{\sqrt{\det(A)}}\begin{pmatrix}
                                         a&b\\c&d
                                 \end{pmatrix} \in \SL{\BK}
$$ 
such that both $A$ and $A^\prime$ define the same map. This map is called a 
\textbf{$\BK$-M\"{o}bius map}. It is clear that $\BP(\BK)$ is the orbit of $[1:0]$ with 
respect to $\SL{\BK}.$ This implies that $\SL{\BK}$ acts transitively on $\BP(\BK).$

If $cx+dy\neq0$ then the map 
$$
\begin{pmatrix}
       a&b\\c&d
\end{pmatrix}:[x:1]\mapsto [ax+b:cx+d] \sim \left[\frac{ax+b}{cx+d}:1\right]
$$
can be abbreviated as $g(x)=\frac{ax+b}{cx+d}$ \citelist{ \cite{MR2977041}*{Ch. 2}\cite{Kisil14a}\cite{Kisil15a}}. In the following, this formula will be used as a notation for more accurate discussion in terms of the projective line.

Given a M\"{o}bius map $g,$ $\tr^2(g)$ is defined to be $\tr^2(g)=(a+d)^2,$ where the representative matrix of $g$ is
$A= \begin{pmatrix}
           a&b\\c&d
     \end{pmatrix}$ in $\SL{\BK}.$ % that represents $g.$ 

Suppose that $g$ is not the identity map. Then, eigenvalues of $A$ are solutions 
$$
\lambda_{1,2}= \frac{d+a\mp\sqrt{(d+a)^2-4}}{2}
$$
of the quadratic characteristic equation with $(a+d)^2=\tr^2 A$ being the principal 
part of the discriminant. An eigenvector 
$\begin{pmatrix}
         x\\y
 \end{pmatrix}$ of $A$ corresponds to a fixed point $[x:y]$ of $g.$ Then, we can classify 
M\"{o}bius maps through the eigenvalues of~$A.$ 
\begin{enumerate}
\item $A$ has two different complex-conjugated eigenvalues if and only if $0\leq\tr^2(g)<4.$ 
That means, $g$ fixes two distinct complex-conjugated points in $\BP(\BC)$ and fixes no point in $\BP(\BR).$ Such map is called \textbf{elliptic}.

\item $A$ has a double eigenvalue if and only if $\tr^2(g)=4.$ That means, $g$ fixes a double point. Such a map is called \textbf{parabolic}.

\item $A$ has two distinct real eigenvalues if and only if $\tr^2(g)>4.$ That means, $g$ fixes two distinct points. Such a map is called \textbf{hyperbolic}.

\item For $\BK=\BC,$ there is an extra class as follows. $A$ has two distinct non-real eigenvalues if and only if $\tr^2(g)\notin [0,\infty).$ In other words, $A$ has two distinct complex eigenvalues if $\Imm(\sqrt{\tr^2(g)})\neq0.$ That means, $g$ fixes two distinct complex points. Such a map is called \textbf{strictly loxodromic}.
\end{enumerate} 

The last type of transformation is not possible for $\BK=\BR.$ The class, which contains the classes of hyperbolic and strictly loxodromic maps, is called the \textbf{class of loxodromic maps}.

Obviously, $\SL{\BC}$ is the disjoint union of $\{I\}$ and the above four classes (parabolic, elliptic, hyperbolic and strictly loxodromic) of maps. $\SL{\BR},$ which is a subgroup of 
$\SL{\BC},$ splits into the disjoint union of $\{I\}$ and the three classes of parabolic, elliptic and hyperbolic maps. It is important that continuous one-parameter subgroups of 
$\SL{\BK}$ consist only of maps of the same type \citelist{\cite{MR2977041}*{Ch. 3}\cite{Beardon95}*{Ch. 4}\cite{Jacquesshort}}.

%%%%%%%%%%%%%%%%%%%%%%
%%%%%%%%%%%%%%%%%%%%%%
\subsection{Generalisation of M\"{o}bius Transformations and Clifford Algebras}
The importance of M\"{o}bius transformations prompts their generalisation from $\BR^2$ to $\BR^n$ or even some pseudo-Euclidean spaces. Clifford algebras $\cl(n)$ allow us to generalise M\"{o}bius transformations from 
$\BR^2$ to $\BR^n$ \cite{MR1834977}*{Ch. 9}. Afterwards, Clifford algebras $\cl(p,q)$ allow us to consider 
M\"{o}bius maps on a pseudo-Euclidean space $\BR^{p,q}.$ 

Clifford algebra $\cl(V,Q)$ is an algebra generated by a real vector space~$V$ with a quadratic form~$Q$. If 
$$
Q(x)=x_1^2+\cdots +x_p^2-x_{p+1}^2-\cdots -x_{p+q}^2,
$$
then notation $\cl(V,Q)$ is replaced with $\cl(p,q)$ where $n=p+q$ is the dimension of 
the vector space. Usually generators of $\cl(p,q)$ are denoted by 
$e_j\,(1\leq j\leq n),$ and they satisfy relations 
\begin{gather*}
e_ie_j = -e_je_i \; \text{when $i\neq j$},\\
e_i^2 =-1 \; \text{when $1\leq i\leq p$} \quad \text{and} \quad
e_i^2 =1 \; \text{when $p+1\leq i\leq p+q$}.
\end{gather*}
For further study, see 
\citelist{\cite{MR1834977}*{Ch. 9}\cite{MR1369094}*{Ch. 15}\cite{MR962440}\cite{MR1044769}\cite{MR859856}\cite{MR1199582}\cite{MR846490}\cite{MR780036}
\cite{zbMATH03028290}\cite{zbMATH02588269}\cite{MR2724360}\cite{MR803326}*{Ch. 10}}.

Note that complex numbers can be recovered as an even sub-algebra of $\cl(2,0),$ that is,  the sub-algebra spanned by $\{1,e_1e_2\}.$ Thus, the theory of complex M\"{o}bius maps can be obtained as a special  case of general Clifford algebra setup. The corresponding even sub-algebra of $\cl(1,1)$ is isomorphic to double numbers (see below). Thus, the main topic of this paper can be viewed as investigation of special cases of lower dimensional Clifford algebras. Our consideration of dual numbers brings a new case of degenerate quadratic forms, which are usually omitted in the Clifford algebra framework. Indeed, the presence of a nilpotent element $(\epsilon^2=0)$ makes many standard tools unsuitable. This paper presents some initial steps in this case.
%%%%%%%%%%%%%%%%%%%%%%%%%%%
%%%%%%%%%%%%%%%%%%%%%%%%%%%%%
\subsection{Double and Dual Projective Lines and M\"{o}bius Transformations}
Double (resp. dual) numbers form a two-dimensional commutative algebra over~$\BR$ spanned 
by~$1$ and $j$ (resp. $\epsilon$) such that $j^2=1$ (resp. $\epsilon^2=0$). \citelist{\cite{MR2977041}*{App. A1}\cite{ShirokovSymplectic}}. The set of all double (resp. dual) numbers is denoted by $\fpower{\BR}$ (resp.~$\BD$). It is known that any two-dimensional commutative algebra over $\BR$ is isomorphic to either $\BC,$ $\BD,$ or $\fpower{\BR}$ \citelist{\cite{MR3029360}\cite{MR2977041}*{App. A1}}. $\BD$ and $\fpower{\BR}$ are interesting complements to the field $\BC$ because they contain nilpotent and idempotent elements and are the simplest models for more complicated rings. The hypercomplex number systems are strongly connected to the theory of Clifford algebras and Lie groups \cites{Ulrych,MR1917405,MR1845956,MR1785355,MR1369094,MR2201799,MR2683557,
PilipchukAndrianovMarkert16a,Pilipchuk10a,MR3362187,MR1784864,MR2063579}. Algebraic properties of higher dimensional geometric spaces can be investigated in terms of hypercomplex matrix representations of Clifford algebras \cite{Ulrych}. Our main aim is to investigate M\"{o}bius transformation over these two commutative algebras. % over $\BR.$ %rings of double and dual numbers.
%Furthermore, some of our results are true for a general ring $R.$

%%%%%%%%%
Let $\BA$ be a ring of complex, double or dual numbers. Let $\sim$ be an equivalence relation on $\BA^2{\setminus}\{(0,0)\}$ defined as follows: $(z_1,z_2){\sim}(z_3,z_4)$ if and only if there exists a unit (an invertible element) $u\in\BA$ such that $z_1=uz_3$ and 
$z_2=uz_4.$ The set of all equivalence classes is denoted by $\BA^2/\sim.$  The point of 
$\BA^2/{\sim}$ corresponding to a vector $\begin{pmatrix} x\\y \end{pmatrix}$ is denoted by $[x:y].$ Therefore, $\BA/{\sim}$ contains the following two types of equivalence classes:
\begin{enumerate}
\item $[x:y]$ such that $x\BA+y\BA=\BA.$
\item $[x:y]$ such that $x\BA+y\BA\neq\BA.$
\end{enumerate}

Equivalence classes of the first type are points of the projective line over $\BA$ which is denoted by $\BP(\BA).$ There is a natural embedding $r:x\mapsto[x:1]$ of $\BA$ into the projective line. For $\fpower{\BR}$ and $\BD,$ $\BP(\BA){\setminus}r(\BA)$ has more than one ideal element \cite{MR520230}*{Suppl. C}. 
 
A linear transformation of $\BA^2$ can be represented by a $2\times2$-matrix 
$\begin{pmatrix}a&b\\c&d \end{pmatrix}$ applied to two-dimensional vectors. The transformation is not degenerate if $ad-bc$ is a unit in $\BA.$ The collection of all $2\times 2$-matrices, $A=\begin{pmatrix}a&b\\c&d\end{pmatrix},$ such that $ad-bc$ is a unit, is a group denoted by $\GL{\BA}.$ The collection of all 
$2\times 2$-matrices, $A=\begin{pmatrix}a&b\\c&d\end{pmatrix}$, such that $ad-bc=1$ is a group denoted by 
$\SL{\BA}$. It is a subgroup of $\GL{\BA}.$  A linear map $\BA^2\rightarrow \BA^2$ is a class invariant for ${\sim}.$ Therefore, it produces the map $\BA^2/{\sim}\rightarrow\BA^2/{\sim}$ as follows:
\begin{align}\label{moba}
\begin{pmatrix}a&b\\c&d\end{pmatrix}[x:y]=[ax+by:cx+dy],\quad 
\begin{pmatrix}a&b\\c&d\end{pmatrix}\in \GL{\BA}.
\end{align}
For any unit $u\in \BA,$ $\begin{pmatrix}a&b\\c&d \end{pmatrix}$ and 
$\begin{pmatrix} ua&ub\\uc&ud \end{pmatrix}$ define the same 
map \eqref{moba}. % \cite{MR2977041}*{Ch. 3}. 
This map is called a \textbf{$\BA$-M\"{o}bius map}. Recall that, $\BK/{\sim}$ is the only 
$\GL{\BK}$-orbit. In contrast, $\BA/{\sim}$ has more than one $\GL{\BA}$-orbit.
$\GL{\BK}$ and $\SL{\BK}$ have the same number of types of non-equivalence connected continuous one-parameter subgroups. While the number of types of non-equivalence connected continuous one-parameter subgroups of $\GL{\BA}$ and $\SL{\BA}$ is different for 
$\BA=\BD.$ This shows that the theory of M\"{o}bius maps for double and dual numbers is different from $\BC$ and deserves a special consideration.

A wider context for our work is provided by the Erlangen programme of F.~Klein, cf.~\cites{MR3220526,MR2977041}. Similarly to the case of $\SL{\BR}$~\cites{MR2725007,MR2361159,MR2977041}, we want to characterise all 
non-equivalent homogeneous spaces $G/H$, where $G$ is one of the groups $\GL{\fpower{\BR}}$, $\GL{\BD}$, $\SL{\fpower{\BR}}$, $\SL{\BD},$ and $H$ is a closed continuous subgroup of $G$. The natural action of $G$ on the homogeneous space $G/H$ is geometrically represented by M\"{o}bius transformations. The respective conformal geometry is intimately connected with various physical models~\citelist{\cite{MR2337687}\cite{MR2977041}\cite{Kisil17a}\cite{MR520230}*{Suppl. C}}. Geometrical language provides an enlightening environment for many related questions, e.g., continued fractions~\cite{MR1744876}, analytic functions~\cites{MR1627084,MR1744876}, 
spectral theory~\cites{MR2098877,MR2001578,MR3220526}, etc.
%%%%%%%%%%%%%5
%If $cx+dy$ a unit then double (dual) M\"{o}bius map can be expressed  as $g(z)=\frac{az+b}{cz+d}.$%\quad \text{for some } a,b,c,d \in \BK\text{ and }  ad-bc=1. $$ 

\subsection{Main Results and the Outline}
This paper considers the following four questions. For $\BA={\fpower{\BR}}$ or $\BD:$
\begin{enumerate}
\item How many $\GL{\BA}$-orbits does $\BA^2/{\sim}$ have?
\item Is there a useful classification of $\BA$-M\"{o}bius map based on fixed points?
\item How many types of continuous one-parameter subgroups does $\GL{\BA}$ have?
\item Do we lose any type of continuous one-parameter subgroups when we move from $\GL{\BA}$ to $\SL{\BA}$? 
\end{enumerate}
Our study prompts a conclusion that M\"{o}bius transformations of $\fpower{\BR}$ and $\BD$ have many similarities to M\"{o}bius transformations of $\BR$ \cite{MR2977041} rather than $\BC$. This is essentially due to the fact, $\fpower{\BR}$ and $\BD$ are not algebraically closed. In particular square root, that is solutions of $x^2-u=0,$ does not exist.

The outline of the paper is as follows.

In the next section, we provide preliminary material about the rings of double and dual numbers. 

In the third section, we discuss the general linear groups and their action on the cosets $\BA^2/{\sim}.$ In particular, we will see in Propositions~\ref{po1} and~\ref{pdd1} that the projective lines do not coincide with $\BA^2/{\sim}$ (in contrast to the cases of the fields).

In the fourth section, we define M\"{o}bius transformations over $\BP(\BA)$ where $\BA$ is the ring of double or dual numbers. Lemma~\ref{orbitr11} states that $\SL{\BA}$ acts transitively on $\BP(\BA).$ Proposition~\ref{po} (resp. Proposition~\ref{pd})  shows that the action of $\GL{\fpower{\BR}}$ (resp. $\GL{\BD}$) over $\fpower{\BR}^2/{\sim}$ (resp. $\BD^2/\sim$) has more than one orbit. 

The last section contains our main results. In this section, relying on the four types of continuous one-parameter subgroups of $\SL{\BR},$ we build, up to similarity and rescaling, the different types of continuous one-parameter subgroups of $\GL{\fpower{\BR}},$ $\SL{\fpower{\BR}},$ $\GL{\BD}$ and $\SL{\BD}.$ Proposition~\ref{oslo1} gives the equation of the orbit of an arbitrary element in $\BP(\fpower{\BR})$ concerning one of the continuous one-parameter subgroups of $\SL{\fpower{\BR}}.$ Proposition~\ref{sldorbit} gives the equation of the orbit of an arbitrary element in $\BP(\BD)$ concerning one of the connected continuous one-parameter subgroups of~$\GL{\BD}.$  

%%%%%%%%%%%%%%%% preliminarily %%%%%%%%
%%%%%%%%%%%%%%%%%
%%%%%%%%%%%%%%%%%
\section{Preliminaries}
We follow notation and terminology on algebraic structures like rings, ideals, etc. of \cites{MR1245487,RomanSteven}, to which the reader is referred for all standard notions.
%%%%%%%%%%%%%
%%%%%%%%%%%%%%
\subsection{Double Numbers}
The double numbers $(\fpower{\BR},+,\cdot)$ is a two-dimensional commutative and associative real algebra with unity spanned by $1$ and $j,$ where $j$ has the property $j^2=1.$

If we define $P_\pm= \frac{1}{2}(1\pm j),$ then we can write the set of all double numbers as 
$\fpower{\BR}=\{a_{+}P_{+}+a_{-}P_{-} :a_{+},a_{-} \in \BR\}.$ Therefore, the following properties hold: 
\begin{enumerate}
\item $P_{\pm}^2=P_{\pm},$ $P_{+} P_{-}=0.$
\item Let $a=a_{+}P_{+}+a_{-}P_{-}\in \fpower{\BR}.$ Then,
\begin{itemize} 
\item[(a)] is an invertible double number if and only if $a_{+}a_{-}\neq 0,$ and the inverse of $a$ is $a^{-1}=a_{+}^{-1}P_{+}+a_{-}^{-1}P_{-}.$
\item[(b)] The conjugate $\overline{a}=a_1-ja_2$ of $a=a_1+ja_2$ is 
$\overline{a}=a_{-}P_{+}+a_{+}P_{-}.$
\item[(c)] The square roots of $a,$ that is, solutions of the equation $a=x^2,$ have up to four values:
$$ 
\pm(\sqrt{a_{+}}P_{+}+\sqrt{a_{-}}P_{-})\quad \text{or} \quad 
\pm(\sqrt{a_{+}}P_{+}-\sqrt{a_{-}}P_{-})
$$ if $a_{+},a_{-}\geq 0$ and are not defined otherwise.
\end{itemize}
\end{enumerate}

\begin{remark}\hspace*{1ex}
\begin{enumerate}
\item Since $(\fpower{\BR},+,\cdot)$ is a commutative ring with unity then $0$ and $1$ are idempotent elements. In fact, from the definition of $P_{+}$ and $P_{-}$, we can see that both of them are idempotent elements too. Moreover, $0,$ $1,$ $P_{+},$ and $P_{-}$ are the only idempotents of $\fpower{\BR}.$
\item Because $(\fpower{\BR},+,\cdot)$ is a commutative ring with unity, $0$ is a nilpotent element. Moreover, the ring of double numbers does not contain any non-zero nilpotent elements.
\end{enumerate}
\end{remark}

                        %%######### dual numbers ####################
                                  %%%%%%%%%%%%%%%
                                      %%%%%%%%%%%%
\subsection{Dual Numbers}
The dual numbers $(\BD,+,\cdot)$ is a two-dimensional commutative and associative real algebra with unity spanned by $1$ and $\epsilon,$ where $\epsilon$ has the property 
$\epsilon^2=0.$
\begin{remark} \hspace*{1ex}
\begin{enumerate} 
\item Let $a=a_1+\epsilon a_2\in \BD.$ Then,
\begin{enumerate}
\item The inverse of $a$ is $$a^{-1}=
\begin{cases}
a_1^{-1}-\epsilon (a_1^{-1})^2a_2,& \quad \text{if $a_1\neq 0;$}\\
\text{undefined},&\quad \text{otherwise}.
\end{cases}
$$
\item The square root of $a$ is, that is, solutions of the equation $a=x^2,$  
$$
\sqrt{a}=
\begin{cases}
\pm(\sqrt{a_1}+\epsilon \displaystyle\frac{a_2}{2\sqrt{a_1}}), & \quad \text{if $a_1> 0;$}\\
                                                            0, & \quad \text{if $a_1=a_2=0;$}\\
\text{undefined},&\quad \text{otherwise}.
\end{cases}
$$
\end{enumerate}
\item Because $(\BD,+,\cdot)$ is a commutative ring with unity, $0$ and $1$ are idempotent elements. Moreover, the ring of dual numbers does not contain any other idempotent element. 
\item  $0$ is a nilpotent element in $\BD.$ Furthermore, for all $a\in \BR,$ $\epsilon a$ is a nilpotent element.
\end{enumerate}
\end{remark}
%%%%%%%%%%%% General Linear Group And The Projective Line%%%%%%%%%%%%%%%%%
%%%%%%%%%%%%%%%%%%%%%%%%%%%%%%%%%
\section{General Linear Group and the Projective Line}
Hereafter, $R$ is a commutative associative ring with unity~$1.$ The following notion is at the core of our study:
\begin{definition}\label{glg}
$$
\GL{R}=\left\{
\begin{pmatrix}
a&b\\c&d
\end{pmatrix}:a,b,c,d\in R \text{ and }ad-bc\text{ is invertible in } R \right\}.
$$ 
The multiplication on $\GL{R}$ is defined as the usual multiplication of matrices.
\end{definition}
Clearly, because for any two matrices $A$ and $B$ in $\GL{R},$ we have $\det(AB)=\det(A)\det(B)$ as $\GL{R}$ is closed under multiplication. The identity matrix is the identity element of $\GL{R}.$ Obviously, for any commutative associative ring $R$ with unity, $\GL{R}$ is a group called a \textbf{general linear group of}~$R.$

The following notion will be relevant to our study.
\begin{definition}\label{admiss} 
A pair $(a,b)\in R^2 $ is called \textbf{admissible}\index{admissible} if there exist $c,d \in R$ such that 
$\begin{pmatrix}
a&b\\c&d
\end{pmatrix}$ 
is an invertible matrix~\cite{MR1809553}.
\end{definition}
\begin{remark} Let $(a,b)\in R^2$ be an admissible pair and $c,d\in R.$ If 
$\begin{pmatrix}
a&b\\c&d
\end{pmatrix} \in \GL{R}$ then $(c,d)$ is an admissible pair too.
\end{remark}
  
Let $\sim$ be the equivalence relation over $R,$ which is defined in section one. An $R$-linear map 
$R^2\rightarrow R^2$ is a class invariant for ${\sim}$ \cite{Kisil15a}. The point of $R^2/{\sim}$ corresponding to a vector $\begin{pmatrix} x\\y \end{pmatrix}$ is denoted by $[x:y].$ By $R^2/{\sim},$ we mean the set of all equivalence classes. Thus, the projective line over a ring $R$ is defined as follows:
\begin{definition}
$\BP(R)=\{[a:b] :a,b\in R^2{\setminus}\{(0,0)\}\, \text{and $(a,b)$ is admissible}\}$ is the \textbf{projective line} over the ring $R$ \cite{MR1809553}.
\end{definition}
The projective line over a ring is an extension of the concept of the projective line over a field.

The following propositions give us two examples of projective lines, \cite{MR520230}*{Suppl. C} gives the result without a proof which we give here.
 
\begin{proposition}\label{po1}
Any element $[a:b]$ of the projective line $\BP(\fpower{\BR})$ belongs to exactly one of the following six distinct classes:
$$
[1:0], \quad [a:1], \quad [1:\lambda P_{+}], \quad [1:\lambda P_{-}], \quad [P_{+}:P_{-}]\quad\text{or}\quad[P_{-}:P_{+}],
$$
where $a\in \fpower{\BR}$ and $\lambda\in \BR{\setminus}\{0\}.$
\end{proposition}
\begin{proof}
The set of all double numbers is a disjoint union of the following three sets: 
$U(\fpower{\BR})$ (the set of all invertible elements), $\{0\},$ and 
$\widetilde{\fpower{\BR}}$ (the set of all zero divisors). We can also divide 
$\widetilde{\fpower{\BR}}$ into two disjoint sets $\{aP_{+}\}$ and $\{aP_{-}\},$ where~$a$ 
is a non-zero real number. Then, we have, in $\fpower{\BR}^2/{\sim},$ the following different types of equivalence classes:
\begin{enumerate} 
\item If $a\in \overline{\fpower{\BR}}$ and $ b=0$, then $(a,b)\in[1:0].$ 

\item If $a\in{\fpower{\BR}}$ and $b\in \overline{\fpower{\BR}}$, then  
$(a,b)\in [\frac{a}{b}:1].$

\item If $a=a_{+}P_{+}+a_{-}P_{-}\in\overline{\fpower{\BR}}$ and 
$b=b_\pm P_{\pm}\in\widetilde{\fpower{\BR}},$ then $(a,b)\in[1:\lambda P_{\pm}],$ where 
$\lambda=\frac{b_{\pm}}{a_{\pm}}.$

\item If  $a=a_1P_{\pm} ,$ $b=b_1P_{\mp} $, then $(a,b)\in [P_{\pm}:P_{\mp}].$

\item If both $a, b\in \{\lambda P_{\pm}\}\cup\{0\}$, then 
$(a,b)\in [a_1 P_{\pm}:b_1P_{\pm}].$
\end{enumerate}
Thus, any element in $\fpower{\BR}^2/{\sim}$ belongs to one of the following different classes:
$$
[1:0], \quad [a:1], \quad [1:\lambda P_{\pm}], \quad [P_{\pm}:P_{\mp}] 
\quad\text{or}\quad[a_1 P_{\pm}:b_1P_{\pm}],
$$
where $a\in \fpower{\BR},$ $\lambda \in \BR{\setminus}\{0\}$ and $a_1,b_1 \in \BR.$ Here, we are going to show that why $[a_1 P_{\pm}:b_1P_{\pm}],$ for all $a_1,b_1\in \BR,$ does not belong to the projective line $\BP(\fpower{\BR}).$ From Definition \ref{glg}, clearly,
\begin{gather*}
I=\begin{pmatrix} 1&0\\0&1\end{pmatrix},\quad
  \begin{pmatrix} a&1\\1&0\end{pmatrix},\quad
	\begin{pmatrix} 1&\lambda P_{\pm}\\0&1\end{pmatrix},\quad
	\begin{pmatrix} P_{\pm}&P_{\mp}\\1&1\end{pmatrix} \in \GL{\fpower{\BR}}
\end{gather*}
while for all $a,b\in \fpower{\BR}$,
$\begin{pmatrix}
       a_1 P_{\pm}&b_1P_{\pm}\\
       a&b
  \end{pmatrix} \notin \GL{\fpower{\BR}}.$ Thus, the pairs $(1,0),$ $(a,1),$ $(1,\lambda P_{\pm}),$ and 
$(P_{\pm},P_{\mp})$ are admissible pairs while the pairs $(a_1P_{\pm}:b_1P_{\pm})$ are not.
\end{proof}
Thus $\fpower{\BR}$ parametrises almost whole projective line over $\fpower{\BR}$ as $[a:1],$ for all $a\in \fpower{\BR},$ except for the classes $[1:0],$ $[P_{+}:P_{-}],$ 
$[P_{-}:P_{+}],$ $[1:\lambda P_{-}]$ and $[1:\lambda P_{+}]$ for all non-zero $\lambda\in\BR$.
\begin{notation}
A suggestive notation for $[a:1]$ (resp. $[1:0]$, $[1:a_1P_{-}],$ $[1:a_1P_{+}],$ $[P_+:P_-],$ $[P_-:P_+]$) is 
$a$ (resp. $\infty,$ $\frac{1}{a_1}\omega_1,$ $\frac{1}{a_1}\omega_2,$ 
$\sigma_1,$ $\sigma_2$), where $a\in \fpower{\BR}$ and $a_1$ is a non-zero real number 
\citelist{\cite{MR2977041}*{Ch. 8}\cite{MR520230}*{Suppl. C}}. In other words, 
$$
\BP(\fpower{\BR})={\fpower{\BR}}\cup\{\infty,\sigma_1,\sigma_2\}\cup\{a_1\omega_1,a_1\omega_2:a_1\in\BR{\setminus}\{0\}\},
$$
for any non-zero real number~$a$ \citelist{\cite{MR2977041}*{Ch. 8}\cite{MR520230}*{Suppl. C}}.
\end{notation}
%%%%%%%%%%%%%%%%%%%%%%%%

\begin{proposition}\label{pdd1}
Any element $[a:b]$ of the projective line, $\BP(\BD),$ belongs to exactly one of the following three classes:
$$
[1:0], \quad [a:1] \quad\text{or}\quad[1:a_1\epsilon],
$$ 
where $a\in \BD$ and $a_1\in \BR{\setminus}\{0\}.$
\end{proposition}
\begin{proof}
Let $U(\BD)$ be the set of all invertible elements in $\BD$, and let 
$\widetilde{\BD}\cup\{0\}$ denote the remaining elements in $\BD.$ 

Then, we have, in $\BD^2/{\sim},$ four different types of equivalence classes:
\begin{enumerate}
\item If $a\in\BD$ and $b\in \overline{\BD},$ then $(a,b)\in[\frac{a}{b}:1].$

\item If $a\in \overline{\BD}$ and $b=0,$ then $(a,b)\in[1:0].$

\item If $a=a_1+\epsilon a_2\in \overline{\BD}$ and $b=\epsilon b_1\in \widetilde{\BD},$ then 
$(a,b)\in[1:\epsilon \lambda],$ where $\lambda=\frac{b_1}{a_1}.$

\item If both $a$ and $b$ in $\widetilde{\BD}\cup\{0\},$ then $(a,b)\in[\epsilon\lambda_1:\epsilon\lambda_2],$ where $\lambda_{1},\lambda_{2}\in\BR.$
\end{enumerate}
Therefore, any element in $\BD^2/{\sim}$ belongs to one of the following four distinct classes:
$$
[a:1], \quad [1:0], \quad [1:\epsilon a_1]\quad{or}\quad[\epsilon \lambda_1:\epsilon \lambda_2],
$$ 
where $a\in\BD,\text{ }a_1$ is a non-zero real number and $\lambda_{1},\lambda_{2}\in\BR.$ %\in \BR.$
Here, we are going to show why the class $[\lambda_1\epsilon:\lambda_2\epsilon]$ does not belong to the projective line. From Definition \ref{glg}, obviously,
\begin{gather*}
I=\begin{pmatrix}1&0\\0&1\end{pmatrix},\quad
  \begin{pmatrix} a&1\\1&0\end{pmatrix},\quad
	\begin{pmatrix} 1&\epsilon a_1 \\0&1\end{pmatrix} \in \GL{\BD},
\end{gather*}
while for all $a,b\in \BD$,
$
\begin{pmatrix}
        \epsilon\lambda_1&\epsilon \lambda_2\\a&b
\end{pmatrix}\notin \GL{\BD}.
$ Therefore, the pairs, $(1,0),$ $(a,1),$ and $(1,\epsilon a_1)$ are admissible while the pairs 
$(\epsilon \lambda_1,\epsilon \lambda_2)$ are not.
\end{proof}
Thus $\BD$ parametrises almost whole projective line of $\BD$ as $[a:1],$ for all $a\in \BD,$ except classes $[1:0]$ and $[1:\epsilon \lambda ]$ for all non-zero $\lambda\in\BR$.
\begin{notation}
 A suggestive notation for $[a:1]$ (resp. $[1:0],$ $[1:\epsilon a_1]$) is~$a$ (resp. 
$\infty,$ $\frac{1}{a_1}\omega)$, respectively, where $a\in \BD$ and $a_1$ is a non-zero real number \citelist{\cite{MR2977041}*{Ch. 8}\cite{MR520230}*{Suppl. C}}. In other words, 
$$
\BP(\BD)=\BD\cup\{\infty\}\cup\{a_1\omega:a_1\in\BR{\setminus}\{0\}\}
$$
\citelist{\cite{MR2977041}*{Ch. 8}\cite{MR520230}*{Suppl. C}}.
\end{notation}
%%%%%%%%%%%%%%%%%%%%%%%% M\"{o}bius Transformations %%%%%%%%%%
%%%%%%%%%%%%%%%%%%%%
%%%%%%%%%%%%%
\section{M\"{o}bius Transformations}
In this section, we study $\BP(R)$ as a $\GL{R}$-homogeneous space.  
\begin{definition} 
Let $f:S\rightarrow S$ be a transformation. We say that $X\subseteq S$ is an 
\textbf{$f$-invariant}\index{invariant} if $f(X)\subseteq X $ \cite{MR2977041}*{Ch. 2}.
\end{definition}
\begin{definition} 
Let $G$ be a group acting on a set $S.$ A subset $X$ of $S$ is called \textbf{$G$-invariant}\index{invariant} if $f(X)\subseteq X $ for all $f\in G.$
\end{definition}
The following lemma shows that $\BP(R)$ is a $\GL{R}$-invariant set.
\begin{lemma}\label{invadm} If 
$\begin{pmatrix}
a&b\\ c&d
\end{pmatrix}\in \GL{R}$ and 
$\begin{pmatrix}
x\\y
\end{pmatrix}\in R^2$ is an admissible pair, then 
$\begin{pmatrix}
a&b\\ c&d
\end{pmatrix}
\begin{pmatrix}
x\\ y
\end{pmatrix}=
\begin{pmatrix}
ax+by\\ cx+dy
\end{pmatrix}$ is an admissible pair too.
\end{lemma}

\begin{proof}
Let 
$\begin{pmatrix}
x\\y
\end{pmatrix} \in R^2$ be an admissible pair and $M=
\begin{pmatrix}
a&b\\ c&d
\end{pmatrix}\in \GL{R}$ such that $M
\begin{pmatrix}
x\\y
\end{pmatrix}=
\begin{pmatrix}
w_1\\w_2
\end{pmatrix}.$ Because 
$\begin{pmatrix}
 x\\y
 \end{pmatrix}$ is an admissible pair, there exists 
$\begin{pmatrix}
 s_1\\s_2
 \end{pmatrix}\in R^2$ such that 
$M^\prime=\begin{pmatrix}
             x&s_1\\ y&s_2
          \end{pmatrix}\in \GL{R}.$ 
 Thus, $M\cdot M^\prime=
\begin{pmatrix}
w_1&t_1\\ w_2&t_2
\end{pmatrix} \in \GL{R},$ where $t_1=as_1+bs_2$ and $t_2=cs_1+ds_2.$ 
\end{proof}

\begin{corollary} If 
$\begin{pmatrix}
 a&b\\ c&d
 \end{pmatrix}\in \GL{R}$ and 
$\begin{pmatrix}
 x\\y
 \end{pmatrix}\in R^2$ is not an admissible pair, then 
$\begin{pmatrix}
 a&b\\ c&d
 \end{pmatrix}
 \begin{pmatrix}
 x\\ y
 \end{pmatrix}=
 \begin{pmatrix}
 ax+by\\ cx+dy
 \end{pmatrix}$ is not an admissible pair too.
\end{corollary}
The proof follows from Lemma~\ref{invadm}. 
\begin{proposition} 
Let $M=\begin{pmatrix}a&b\\c&d\end{pmatrix}\in \GL{R}$ and define a mapping 
$M:\BP(R)\to\BP(R)$ as 
$$
M[x:y]=\begin{pmatrix}
             a&b\\c&d
       \end{pmatrix}[x:y]=[ax+by:cx+dy].
$$ 
Then, this mapping is class-preserving.
\end{proposition}
\begin{proof}
Recall, $[x:y]=[v:w]$ means there exists an invertible element $u\in R$ such that 
$x=uv,$ and $y=uw.$ 
\begin{align*} 
M[x:y]&=\begin{pmatrix}
              a&b\\c&d
        \end{pmatrix}[x:y]=[ax+by:cx+dy]\\
      &=[u(av+bw):u(cv+dw)]=M[v:w]. \tag*{\qed}
\end{align*}
\renewcommand{\qedsymbol}{} 
\end{proof}
\begin{definition} 
Let $M=\begin{pmatrix}a&b\\c&d\end{pmatrix} \in \GL{R}$. Let $T_M:\BP(R)\to \BP(R)$ be a function defined by 
$$
T_M([x:y])=M[x:y]=[ax+by:cx+dy].
$$
The map $T_M$ is called \textbf{M\"{o}bius transformation}\index{M\"{o}bius transformation!}.
\end{definition}
If $cx+dy$ is a unit in $R$ then the map 
$$
\begin{pmatrix}
       a&b\\c&d
\end{pmatrix}:[x:1]\rightarrow[ax+b:cx+d]{\sim}\left[\frac{ax+b}{cx+d}:1\right]
$$ 
can be abbreviated to $g(x)=\frac{ax+b}{cx+d}$ \citelist{\cite{MR2977041}*{Ch. 2}\cite{Kisil14a}\cite{Kisil15a}}. In the following, this formula will be used as a notation for more accurate discussion in terms of the projective line.

Let $\BA$ be one of $\BR,$ $\fpower{\BR},$ or $\BD.$ Let $A,A^\prime\in \GL{\BA}$ be such 
that $A=uA^\prime,$ where $u$ is a unit. If M\"{o}bius transformations $T_A$ and 
$T_{A^\prime}$ are considered, then $T_A=T_{A^\prime}.$ The algebraic structure of $\BA$ shows that for any matrix $A\in \GL{\BA}$ such that $\det(A)=u^2$ and $u$ is an invertible element in $\BA,$ there is $A^\prime=\frac{1}{u}A\in \GL{\BA}$ such that $\det( A^\prime)=1.$  

Thus, we define 
$$
\SL{\BA}=\{A\in \GL{\BA}: \det(A)=1\},
$$ 
which is a subgroup of $\GL{\BA}.$ 

Let $A\in \SL{\BA}$ and let $T_A$ be the M\"{o}bius transformation of $\BP(\BA)$ defined 
by~$A.$ Then,  $A$ is called a \textbf{representative matrix} of $T_A.$ The set of 
all M\"{o}bius maps of $\BP(\BA)$ is denoted by $\Mob{\BA}.$  Clearly, $\Mob{\BA}$ is a group. 

\begin{proposition}
Let $\BA$ be one of $\BR,$ $\fpower{\BR},$ or $\BD.$ Let $\pi : \SL{\BA}\rightarrow \Mob{\BA}$ be a map defined as $\pi(A)=T_A$. Then, $\pi$ is a group homomorphism. 
\end{proposition}
By a direct check, $\pi(AB)=T_{AB}=T_A\circ T_B.$ 

If $\BA=\BR$ then the kernel of $\pi$ is $\{\pm I\}.$ Therefore, $\Mob{\BR}\cong \SL{\BR}/\{\pm I\}$ \citelist{\cite{Beardon01}\cite{olsen2010geometry}}.

If $\BA={\fpower{\BR}}$ then the kernel of $\pi$ is $\{\pm I,\pm jI\}.$ 
Therefore,
$$
\Mob{\fpower{\BR}}\cong \SL{\fpower{\BR}}/\{\pm I,\pm jI\}.
$$ 
If $\BA=\BD$  then the kernel of $\pi$ is $\{\pm I\}.$ Thus, $\Mob{\BD}\cong \SL{\BD}/\{\pm I\}.$

The proof of the next lemma follows immediately from Lemma~\ref{invadm}.
\begin{lemma}\label{orbitr11}
$\GL{R}$ acts transitively on $\BP(R).$
\end{lemma}
\begin{thm}\label{po} 
The set $\fpower{\BR}^2/{\sim}$ is a disjoint union of the following three orbits of $\GL{\fpower{\BR}}$:
\begin{enumerate} 
\item The orbit of $[1:0]$ is the projective line over $\fpower{\BR}.$
\item The orbit of $[P_{+}:0]$ is the set 
$$
P\BR_{+}=\{[\lambda_1P_{+}:\lambda_2P_+]:\lambda_1, \lambda_2\; \text{are real numbers not both $0$}\}.
$$
\item The orbit of $[P_{-}:0]$ is the set 
$$
P\BR_{-}=\{[\lambda_1P_-:\lambda_2P_-]:\lambda_1, \lambda_2\; \text{are real numbers not both $0$}\}.
$$
\end{enumerate}
\end{thm}

\begin{proof}Recall, 
$$
\GL{\fpower{\BR}}=
\left\{\begin{pmatrix}
               a&b\\c&d
       \end{pmatrix}:a,b,c,d\in \fpower{\BR}\; \text{and $ad-bc$ is a unit in $\fpower{\BR}$}\right\}.
$$
(1) Immediate from Lemma~\ref{orbitr11}.

(2) For all $A=\begin{pmatrix}
             a&b\\ c&d
           \end{pmatrix}\in \GL{\fpower{\BR}},$ we have 
$$
A[P_+:0]=[aP_+:cP_+].
$$
Therefore, the orbit of $[P_{+}:0]$ is a subset of $P\BR_{+}.$ Conversely, let 
$[\lambda P_{+}:\mu P_{+}]$ be any element in $ P\BR_{+}.$ Clearly, 
$$
A[P_{+}:0]=[\lambda P_{+}:\mu P_{+}],
$$
where 
$$
A=\begin{pmatrix} \lambda P_{+}&P_-\\\mu P_{+}+P_{-}&P_{+}\end{pmatrix} \quad \text{or} \quad 
  \begin{pmatrix} \lambda P_{+}+P_{-}&P_{+}\\\mu P_{+}&P_{-}\end{pmatrix}.
$$ 
Therefore, $[\lambda P_{+}:\mu P_{+}]$ is in the orbit of $[P_{+}:0],$ i.e., $P\BR_{+}$ is a subset of the orbit of $[P_{+}:0].$ So, the orbit of $[P_{+}:0]$ equals the set $P\BR_{+}.$

(3) By using the same approach as in (2).
\end{proof}
From the previous proposition, we can split $\fpower{\BR}^2/{\sim}$ into three sets: The orbit of $[1:0],$ the orbit of $[P_+:0]$, and the orbit of $[P_{-}:0]$. The next proposition explains an isomorphism between the orbit of $[P_\pm:0]$ and $\BP(\BR).$ 
\begin{proposition}\label{inter}
Let $X$ be the $\GL{\fpower{\BR}}$-orbit of $[P_\pm:0].$ There is a projection 
$p_\pm:\SL{\fpower{\BR}}\rightarrow \SL{\BR}$ defined by 
$$
p_\pm(g)=g_\pm \quad \text{for} \quad g=g_{+}P_{+}+g_{-}P_{-}\in \SL{\fpower{\BR}}
$$ 
and a bijection $f:\BP(\BR)\rightarrow X$ defined by 
$$
f[x:y]= [xP_\pm:yP_\pm] \quad \text{for $x,y\in\BR$}.
$$ 
\end{proposition}
I omit the proof as it is a straightforward result. 

The following lemma, whose proof is by direct calculation, will be useful later.
\begin{lemma}\label{deto}
Let $A_\mp\in \Mat{\BR}.$ If $A=A_{+}P_{+}+A_{-}P_{-}\in \Mat{\fpower{\BR}}$ then 
$$
\det (A)=\det (A_+)P_{+}+\det(A_{-})P_{-}.
$$ 
\end{lemma}
Clearly, $A=A_+P_++A_-P_-\in \SL{\fpower{\BR}}$ if and only if $A_\pm\in \SL{\BR}.$
\begin{thm}\label{pd} 
The set $\BD^2/{\sim}$ is a disjoint union of the following two $\GL{\BD}$-orbits:
\begin{enumerate} 
\item The orbit of $[1:0],$ which is the projective line over $\BD.$
\item The orbit of $[\epsilon a:0],$ which is the set 
$$
P\BR=\{[\epsilon\lambda_1:\epsilon\lambda_2]:\lambda_1,\lambda_2\; \text{are real numbers not both $0$}\}.
$$
\end{enumerate} 
\end{thm}
\begin{proof} Recall, 
$$
\GL{\BD}=
\left\{\begin{pmatrix}
              a&b\\c&d
       \end{pmatrix}:a,b,c,d\in \BD\; \text{and $ad-bc$ is a unit in $\BD$} \right\}.
$$
(1) Immediate from Lemma~\ref{orbitr11}.

(2) Let $A=\begin{pmatrix}
          a&b\\ c&d
       \end{pmatrix}\in \GL{\BD}.$ Then, $A[\epsilon:0]=[\epsilon a:\epsilon c].$ Therefore, the orbit of $[\epsilon:0]$ is a subset of $P\BR.$ Conversely, let $[\epsilon\lambda:\epsilon\mu]$ be any element 
in~$P\BR.$ Clearly, 
$$
A[\epsilon:0]=[\epsilon\lambda:\epsilon\mu],
$$
where $A=\begin{pmatrix}\lambda+\epsilon&-\mu\\\mu+\epsilon&\lambda\end{pmatrix},$ i.e., 
$P\BR$ is a subset of the orbit of $[\epsilon:0].$ So, the orbit of $[\epsilon:0]$ equals the set $P\BR.$
\end{proof}
As a consequence of the above proposition,  we see that $\BD^2/{\sim}$ splits into two sets: the orbit of $[1:0]$ and the orbit of $[\epsilon :0].$ The next proposition explains an isomorphism between the orbit of $[\epsilon:0]$ and $\BP(\BR).$
\begin{proposition}\label{inter1}
Let $X$ be the $\GL{\BD}$-orbit of $[\epsilon:0].$ There is a projection 
$p:\SL{\BD}\rightarrow \SL{\BR}$ defined by 
$$
p(g)=g_1 \quad \text{for} \quad g=g_1+\epsilon g_2\in \SL{\BD}
$$
and a bijection $f:\BP(\BR)\rightarrow X$ defined by, for all $[x:y]\in\BP(\BR),$ 
$$
f[x:y]=[\epsilon x:\epsilon y] \quad\text{for $x,y\in\BR$}.
$$ 
\end{proposition}
I omit the proof as it is a straightforward result.
\begin{definition}
Let $A=\begin{pmatrix}a&b\\c&d\end{pmatrix}\in \Mat{R}.$ We define $\widehat{A}$ as follows: 
$$
\widehat{A}=\begin{pmatrix}d&-b\\-c&a\end{pmatrix}.
$$
\end{definition}
The following lemma, whose proof is by direct calculation, will be useful later. 
\begin{lemma} \label{detd} 
Let $A_1,A_2\in \Mat{\BR}.$ If $A=A_1+\epsilon A_2\in \Mat{\BD}$ then 
$$
\det (A)=\det (A_1)+\epsilon \tr(A_1\widehat{A}_2).
$$ 
\end{lemma}
Clearly, $A=A_1+\epsilon A_2 \in \SL{\BD}$ if and only if $A_1\in \SL{\BR}$ and $\tr(A_1\widehat{A}_2)=0.$

%%%%%%%%%%%%%%% Continuous One-Parameter Subgroups %%%
%%%%%%%%%%%%%%%%%%%%%%%%%
%%%%%%%%%%%%%%%%%%%%%%

\section{Continuous One-Parameter Subgroups}
This section investigates, up to similarity and rescaling, the number of different types of continuous 
one-parameter subgroups of $\GL{\fpower{\BR}},$ $\SL{\fpower{\BR}},$ $\GL{\BD}$ and $\SL{\BD}.$ 
\begin{definition} 
A \textbf{continuous one-parameter group} is a group homomorphism $\phi:\BR\to G$, where $G$ is a topological group, and we have 
\begin{enumerate}
\item $\phi(t_1+t_2)=\phi(t_1)\cdot\phi(t_2),$ for $t_1,t_2\in \BR,$
\item $\phi(0)=e$ where $e$ is the identity element in $G.$
\end{enumerate}
\end{definition} 

\begin{lemma}
Let $g_t$ be a continuous one-parameter subgroup. Let $t_0\neq0.$ If $z$ is a fixed point of $g_{t_0}$ then $g_t$ fixes~$z$ for any~$t.$
\end{lemma}
\begin{proof} 
Let $t_0\neq0.$ Let $z$ be a fixed point of $g_{t_0},$ that is, $g_{t_0}(z)=z.$ Let us assume that 
for some $t_1,$ $g_{t_1}$ does not fix~ $z.$ Let $g_{t_1}(z)=z^\prime.$ Because $g_t$ is a continuous 
one-parameter subgroups, $g_{t_0}(z^\prime)=g_{t_0}g_{t_1}(z)=g_{t_1}g_{t_0}(z)=z^\prime.$ Therefore, 
$z^\prime$ is another fixed point of $g_{t_0}.$ By repeating this step, we find that $g_{t_0}$ fixes an infinite number of points which implies $t_0=0,$ and this is a contradiction to our assumption. 
\end{proof}
It is shown that there are only four different equivalence classes, up to similarity and rescaling, of continuous one-parameter subgroups of $\SL{\BR}$ \citelist{\cite{MR2977041}*{Ch. 3}\cite{Kisil15a}}. Therefore, there are only the following four types of continuous one-parameter subgroups 
of $\GL{\BR}$ (up to similarity and rescaling): 
\begin{align} 
A^\prime(t)
&=e^{\lambda t}\begin{pmatrix}
    \cosh t&\sinh t\\\sinh t&\cosh t
      \end{pmatrix},\label{Aprime}\\
N^\prime(t)
&=e^{\lambda t}\begin{pmatrix}
       1&0\\t&1
      \end{pmatrix},\label{Nprime}\\
K^\prime(t)
&=e^{\lambda t}\begin{pmatrix}
       \cos t& -\sin t\\ \sin t&\cos t
      \end{pmatrix},\label{Kprime}\\
I^\prime(t)
&=e^{\lambda t}\begin{pmatrix}1&0\\0&1\end{pmatrix}.
\end{align} 

To give a unified form for these classes, that is work for all continuous one-parameter subgroups of  
$\GL{\BR},$ we let $\sigma \in \{-1,0,1\}$ and introduce the following notation \cite{MR2977041}*{Ch. 9}:
%\begin{enumerate}
% \item 
\[ \cos_\sigma t =
  \begin{cases}
    \cos t,  & \quad \text{if $\sigma=-1$;}\\
    1,       & \quad \text{if $\sigma=0$;}\\
    \cosh t, & \quad \text{if $\sigma=1$.}
  \end{cases}
\]
%\item 
\[ \sin_\sigma t =
  \begin{cases}
    \sin t,  & \quad \text{if $\sigma=-1$;}\\
    t,       & \quad \text{if $\sigma=0$;}\\
    \sinh t, & \quad \text{if $\sigma=1$.}
  \end{cases}
\]
%\item 
\[ \tan_\sigma t =
  \begin{cases}
    \tan t,  & \quad \text{if $\sigma=-1$;}\\
    t,       & \quad \text{if $ \sigma=0$;}\\
    \tanh t, & \quad \text{if $\sigma=1$.}
  \end{cases}
\]
%\item %First and second ones give us 
With this notation, formulas (\ref{Aprime},\ref{Nprime},\ref{Kprime}) can be written as:
$$
H_\sigma^\prime(t)=e^{\lambda t}\begin{pmatrix} 
                                \cos_\sigma t & \sigma\sin_\sigma t\\ 
																\sin_\sigma t &\cos_\sigma t
																\end{pmatrix}=
\begin{cases}
K^\prime(t), & \text{if $\sigma=-1$;}\\
N^\prime(t), & \text{if $\sigma=0$;}\\
A^\prime(t), & \text{if $\sigma=1$.}
\end{cases}
$$
%\end{enumerate}
Clearly, there are, up to similarity and rescaling, the following three types of non-trivial continuous one-parameter subgroups of $\SL{\BR}$ \cite{MR2977041}*{Ch. 3},
$$
H_\sigma(t)=\begin{pmatrix} 
            \cos_\sigma t & \sigma\sin_\sigma t\\ 
						\sin_\sigma t &\cos_\sigma t
						\end{pmatrix}=
\begin{cases}
K(t), & \text{if $\sigma=-1$;}\\
N(t), & \text{if $\sigma=0$;}\\
A(t), & \text{if $\sigma=1$.}
\end{cases}
$$

%%%%%%%%%%%%%%5
%%%%%%%%%%%%%%%
\subsection{Continuous One-Parameter Subgroups of $\GL{\fpower{\BR}}$ and $\SL{\fpower{\BR}}$}
We are going to classify all different types of connected continuous one-parameter subgroups in $\GL{\fpower{\BR}}$ and 
$\SL{\fpower{\BR}}.$ Our main technique is the $(P_{+},P_{-})$ decomposition.

\begin{proposition}\label{slo} 
Let $B_+(t)$ and $B_-(t)$ be two one-parameter subsets of $\GL{\BR}$ and 
$$
B(t)=B_+(t)P_++B_-(t)P_-
$$
be the corresponding one-parameter subset of $\GL{\fpower{\BR}}.$ $B(t)$ is a continuous one-parameter subgroup of $\GL{\fpower{\BR}}$ if and only if both $B_+(t)$ and $B_-(t)$ are continuous one-parameter subgroups of $\GL{\BR}$. Furthermore, $B(t)$ is non-trivial if and only if at least one of $B_+(t)$ or $B_-(t)$ is non-trivial.
\end{proposition}
\begin{proof}[Necessity]
%\begin{enumerate}
%\item For necessity. 
Let $B(t)=B_+(t)P_++B_-(t)P_-$ be a continuous one-parameter subgroup of 
$\GL{\fpower{\BR}}.$ Hence, $B(t_1+t_2)=B(t_1)B(t_2)$ for any two real numbers $t_1,t_2.$  Therefore, 
\begin{align*} 
B_+(t_1)B_+(t_2)P_++B_-(t_1)B_-(t_2)P_-&=B_+(t_1+t_2)P_++B_-(t_1+t_2)P_-,\\
B(t_1)B(t_2)&=B_+(t_1+t_2)P_++B_-(t_1+t_2)P_-,
\end{align*}
which means 
$$
B_+(t_1)B_+(t_2)=B_+(t_1+t_2) \quad \text{and} \quad 
B_-(t_1)B_-(t_2)=B_-(t_1+t_2).
$$
So, both $B_+(t)$ and $B_-(t)$ are continuous one-parameter subgroups.
\renewcommand{\qedsymbol}{}
\end{proof}
\begin{proof}[Sufficiency]
Let $B_+(t),B_-(t)$ be two continuous one-parameter subgroups of 
$\GL{\BR}.$ This means that, for any two real numbers $t_1,t_2,$ 
$$
B_+(t_1)B_+(t_2)=B_+(t_1+t_2) \quad \text{and} \quad 
B_-(t_1)B_-(t_2)=B_-(t_1+t_2).
$$
This leads to  
\begin{align*} 
B_+(t_1)B_+(t_2)P_++B_-(t_1)B_-(t_2)P_-&=B_+(t_1+t_2)P_++B_-(t_1+t_2)P_-,\\
B(t_1)B(t_2)&=B_+(t_1+t_2)P_++B_-(t_1+t_2)P_-,
\end{align*}
i.e., $B(t_1)B(t_2)=B(t_1+t_2).$ Thus, $B(t)$ is a continuous one-parameter subgroup of 
$\GL{\fpower{\BR}}.$
%\end{enumerate}
Let $B(t)=B_+(t)P_++B_-(t)P_-$ be a non-trivial continuous one-parameter subgroup of 
$\GL{\fpower{\BR}}$ and $B_+(t),B_-(t)$ be two trivial continuous one-parameter subgroups of $\GL{\BR}.$ That means that 
$$
B(t)=\hat{I}(t)P_++\hat{I}(t)P_-.
$$ 
Thus, $B(t)$ is a trivial continuous one-parameter subgroup of $\GL{\fpower{\BR}}.$ This is a contradiction of our assumption. Thus, at least one of $B_+(t)$ or $B_-(t)$ is a non-trivial continuous one-parameter subgroup of $\GL{\BR}.$ The opposite statement is obvious as well: if at least one of 
$B_+(t)$ or $B_-(t)$ is not trivial then $B(t)=B_+(t)P_++B_-(t)P_-$ is not trivial.
\end{proof} 
By direct calculations, we prove our next proposition. 
\begin{proposition}\label{slo2} Let 
$$
B(t)=B_+(t)P_++B_-(t)P_- \quad \text{and} \quad 
\tilde{B}(t)=\tilde{B}_+(t)P_++\tilde{B}_-(t)P_-
$$ 
be two continuous one-parameter subgroups of $\GL{\fpower{\BR}}.$ $B(t)$ is similar to $ \tilde{B}(t)$ if and only if there exist $K_+,K_-\in \GL{\BR}$ such that 
$$
K_+B_+(t)K_+^{-1}=\tilde{B}_+(t) \quad \text{and} \quad 
K_-B_-(t)K_-^{-1}=\tilde{B}_-(t).
$$
In such a case, 
$$
B(t)=(K_+(t)P_++K_-(t)P_-)\tilde{B}(t)(K^{-1}_+(t)P_++K^{-1}_-(t)P_-).
$$
\end{proposition}
Short calculation shows that a function $f:\GL{\fpower{\BR}}\rightarrow \GL{\fpower{\BR}}$ defined as 
$$
f(X_+(t)P_++X_-(t^\prime)P_-)=X_-(t^\prime)P_++X_+(t)P_-,
$$ 
is a group homomorphism.
%I
\begin{thm}
Any continuous one-parameter subgroup of $\GL{\fpower{\BR}}$ has, up to similarity and rescaling, the following form
$$ 
H_+(t)P_++H_-(at)P_-,
$$ 
where $H_\pm$ is a subgroup similar to  $H^\prime_{\sigma_\pm}$ for 
$\sigma_\pm\in\{-1,0,1,r\}$ and $H_r=I^\prime.$ 
\end{thm}
\begin{proof}
%\begin{enumerate}
%\item 
If $B(t)$ is a trivial continuous one-parameter subgroup then 
$$
B(t)=I^\prime P_++I^\prime P_-.
$$
%\item
Let $B(t)$ be a non-trivial continuous one-parameter subgroup. Then either $B_+(t)$ or $B_-(t)$ does not equal 
$I^{\prime}.$ 
%%%typo: probably instead of \hat{I} it should be I^{\prime}
\begin{enumerate}
\item If $B_+(t)\neq I^\prime,$ then up to scaling it is similar to $H^\prime_{\sigma _+}(t),$ where $\sigma\in\{-1,0,1\}.$ Then $B_-(t)$ is either $I^\prime$ or, up to rescaling with $a\neq0$ and similarity, $H^\prime_{\sigma _-}(t),$ where $\sigma\in\{-1,0,1\}.$ 
\item The case of $B_-(t)\neq I^\prime$ is treated the same way.
\end{enumerate}
%\end{enumerate}
%\renewcommand{\qedsymbol}{}
\end{proof}
\begin{corollary}
Any continuous one-parameter subgroup of $\SL{\fpower{\BR}}$ has, up to similarity and rescaling, the following form
$$ 
H_+(t)P_++H_-(at)P_-,
$$ 
where $H_\pm$ is a subgroup similar to  $H_{\sigma_\pm}$ for $\sigma_\pm\in\{-1,0,1,r\}$ and $H_r=I.$
\end{corollary}
\begin{proof} 
%\begin{enumerate}
%\item 
If $B(t)$ is a trivial continuous one-parameter subgroup then 
$$
B(t)=IP_++IP_-.
$$
%\item 
Let $B(t)$ be a non-trivial continuous one-parameter subgroup. Then either $B_+(t)$ or $B_-(t)$ does not equal to $I.$ 
\begin{enumerate}
\item If $B_+(t)\neq I,$ then up to scaling it is similar to $H_{\sigma _+}(t),$ where $\sigma\in\{-1,0,1\}.$ Then $B_-(t)$ is either $I$ or, up to rescaling with $a\neq0$ and similarity, $H_{\sigma _-}(t),$ where $\sigma\in\{-1,0,1\}.$ 
\item The case of $B_-(t)\neq I$ is treated in the same way.
\end{enumerate}
%\end{enumerate}
%\renewcommand{\qedsymbol}{}
\end{proof}
Clearly, we do not lose any interesting types of connected continuous one-parameter subgroups  when we move from $\GL{\fpower{\BR}}$ to $\SL{\fpower{\BR}}.$

From the previous, there are, up to similarity and rescaling, the following types of isomorphic non-trivial continuous one-parameter subgroups of $\SL{\fpower{\BR}}$:
\begin{enumerate}
\item $ B(t)=H_{\sigma _+}(t)P_++H_{\sigma _-}(at)P_-,$ where $a$ is a positive real number;
\item $ B(t)=H_\sigma (t)P_++IP_-,$  where $a=0;$
\end{enumerate} 
and, $H_{\sigma_+}(t)$ (resp. $H_{\sigma_-}(at)$) is similar to one of $A(t),$ $N(t)$ or $K(t)$ 
(resp. $A(at),$ $N(at)$ or $K(at)$).
 
Let $f=[y_+P_++y_-P_-:1]$ be an arbitrary non-fixed point in $\BP(\fpower{\BR}).$ Next proposition gives the orbit of $f$ concerning the continuous one-parameter subgroup 
$H_{\sigma_+}(t_+)P_++H_{\sigma_-}(t_-)P_-$ where $t_-=at_+.$
\begin{proposition}\label{oslo11}
Let $B(t)=H_{\sigma_+}(t)P_++H_{\sigma_-}(at)P_-$ be a continuous one-parameter subgroup of 
$\SL{\fpower{\BR}},$ where $a$ is a non-zero real number. If a point $[u+jv:1]\in\BP(\fpower{\BR})$ belongs to the $B(t)$-orbit of a point $[y_+P_++y_-P_-:1]\in \BP(\fpower{\BR}),$ then 
$$
\tan_{\sigma_+}^{-1}\frac{y_+-(u+v)}{y_+(u+v)-\sigma_+}= 
\frac{1}{a}\tan_{\sigma_-}^{-1}\frac{y_--(u-v)}{y_-(u-v)-\sigma_-}.
$$
\end{proposition}
\begin{proof} 
Let $a$ be a non-zero real number. 
\begin{multline*}
B(t)[y_+P_++y_-P_-:1]=\\
\left[
\frac{y_++\sigma_+\tan_{\sigma_+}(t)}{y_+\tan_{\sigma_+}(t)+1}P_++
\frac{ y_-+\sigma_-\tan_{\sigma_-}(at)}{y_-\tan_{\sigma_-}(at)+1}P_-:1\right].
\end{multline*}
Let us define 
$$
u^\prime=\frac{y_++\sigma_+\tan_{\sigma_+}(t)}{y_+\tan_{\sigma_+}(t)+1},\qquad 
v^\prime=\frac{ y_-+\sigma_-\tan_{\sigma_-}(at)}{y_-\tan_{\sigma_-}(at)+1}.
$$ 
A simple calculation leads to 
$$
t=\tan_{\sigma_+}^{-1}\frac{y_+-u^\prime}{u^\prime y_+-\sigma_+},\qquad 
at=\tan_{\sigma_+}^{-1}\frac{y_--v^\prime}{v^\prime y_--\sigma_-}.
$$ 
$u^\prime=u+v\text{ and } v^\prime=u-v .$ Thus, we obtain 
$$
\tan_{\sigma_+}^{-1}\frac{y_+-(u+v)}{y_+(u+v)-\sigma_+}= 
\frac{1}{a}\tan_{\sigma_-}^{-1}\frac{y_--(u-v)}{y_-(u-v)-\sigma_-}.
$$
\renewcommand{\qedsymbol}{}
\end{proof}
There are several cases which admit a simpler description. 
\begin{corollary}\label{oslo1}
Let $B(t)=N(t)P_++N(at)P_-$ be a continuous one-parameter subgroup of $\SL{\fpower{\BR}},$ where $a$ is a non-zero real number. If a point 
$$
[u+jv:1]\in\BP(\fpower{\BR})
$$
belongs to the $B(t)$-orbit of a point 
$$
[y_+P_++y_-P_-:1]\in \BP(\fpower{\BR}),
$$
then
$$
u^2-v^2+\frac{(a-1)y_+y_-}{y_+-ay_-}u-\frac{(a+1)y_+y_-}{y_+-ay_-}v=0.
$$ 
\end{corollary}
\begin{proof}
The proof follows immediately from Proposition \ref{oslo11}.
\end{proof}
If one of the components is the identity matrix, then the orbit of 
$$
[y_+P_++y_-P_-:1]\in \BP(\fpower{\BR})
$$
is a line as we are going to see in the next proposition.
\begin{proposition}\label{oslo2}
Let $B(t)=H_\sigma(t)P_++IP_-$ be a continuous one-parameter subgroup of $\GL{\fpower{\BR}},$ where $a$ is a non-zero real number. If a point 
$$
[u+jv:1]\in\BP(\fpower{\BR})
$$ 
belongs to the $B(t)$-orbit of a point $[y_+P_++y_-P_-:1]\in \BP(\fpower{\BR}),$ then 
$$
2v=u^2-v^2-y_-(u-v).
$$
\end{proposition}
\begin{proof}
Clearly, 
$$
B(t)[y_+P_++y_-P_-:1]=
\left[\frac{ y_+\cos_\sigma t+\sigma\sin_\sigma t}{y_+\sin_\sigma t+\cos_\sigma t}P_++y_-P_-:1\right].
$$ 
Let us define 
$$
u^\prime=\frac{ y_+\cos_\sigma t+\sigma\sin_\sigma t}{y_+\sin_\sigma t+\cos_\sigma t},
$$
which means that 
$$
\tan_\sigma t=\frac{y_+-u^\prime}{y_+u^\prime-\sigma} \quad \text{and} \quad v^\prime=y_-.
$$
Since $u^\prime=u+v,\quad v^\prime=u-v,$  therefore
$$
\frac{1}{v^\prime}-\frac{1}{u^\prime}=
\frac{u^\prime-v^\prime}{u^\prime v^\prime}=\frac{2v}{u^2-v^2},
$$ 
and this gives 
$$
\frac{2v}{u^2-v^2}=\frac{u+v-y_-}{y_-(y_-(u+v)}.
$$
A simple calculation yields $2v=u^2-v^2-y_-(u-v).$
\end{proof}
\subsection{Continuous One-Parameter Subgroups of $\GL{\BD}\text{ and }\SL{\BD}$}
This subsection shows that $\GL{\BD}$ has, up to similarity and rescaling, three types of continuous 
one-parameter subgroups associated with a non-trivial M\"{o}bius map.    
\begin{lemma}\label{leA1}   
Let $A(t)$ be a continuous non-trivial one-parameter subgroup of $\GL{\BR},$ and let 
$\sigma\in\{-1,0,1\}$ be such that $A(t)$ similar and re-scalable to~$\hat{H}_\sigma.$ Let $B$ be 
any constant matrix in $\GL{\BR}.$ Then, 
$$
A(s)B=BA(s),
$$
for some $s$ such that $\sin_\sigma(s)\neq0$, if and only if for some $s_0$ and a non-zero real number 
$\lambda,$ we have $B=\lambda A(s_0).$ Therefore, $B$ belongs to the centralizer of $A(t).$
\end{lemma}
\begin{proof}[Necessity] 
Let $B=C^{-1}\begin{pmatrix}
                  a&b\\ c&d
             \end{pmatrix}C\in \GL{\BR}.$ Let 
$$
A(s)=C^{-1}e^{\lambda s}\begin{pmatrix}
                               \cos_\sigma s&\sigma\sin_\sigma s\\ 
															 \sin_\sigma s&\cos_\sigma s
												\end{pmatrix}C
$$ 
be a continuous one-parameter subgroup of $\GL{\BR}.$ Assume that 
$$
BA(s)=A(s)B,
$$
that is,
$$
\begin{pmatrix}
      a&b\\ c&d
\end{pmatrix} 
\begin{pmatrix}
   \cos_\sigma s&\sigma\sin_\sigma s\\
	 \sin_\sigma s&\cos_\sigma s
\end{pmatrix}=
\begin{pmatrix}
   \cos_\sigma s&\sigma\sin_\sigma s\\ 
	 \sin_\sigma s&\cos_\sigma s
\end{pmatrix}
\begin{pmatrix}
     a&b\\ c&d
\end{pmatrix}.
$$ 
Therefore,
%\begin{align}\label{equ1}
%a\cos_\sigma s+b\sin_\sigma s&=a\cos_\sigma s+ c\sigma\sin_\sigma s,
%\end{align}
%\begin{align}\label{equ2}
%a\sigma \sin_\sigma s+b\cos_\sigma s=b\cos_\sigma s+d\sigma  \sin_\sigma s,
%\end{align}
%\begin{align}\label{equ3}
%c \cos_\sigma s+d\sin_\sigma s=a\sin_\sigma s+c \cos_\sigma s,
%\end{align}
%\begin{align}\label{equ4}
%c\sigma \sin_\sigma s+d\cos_\sigma s=b\sin_\sigma s+d \cos_\sigma s.
%\end{align}
\begin{align}
a\cos_\sigma s+b\sin_\sigma s
&=a\cos_\sigma s+ c\sigma\sin_\sigma s,\label{equ1}\\
a\sigma \sin_\sigma s+b\cos_\sigma s
&=b\cos_\sigma s+d\sigma  \sin_\sigma s,\label{equ2}\\
c \cos_\sigma s+d\sin_\sigma s
&=a\sin_\sigma s+c \cos_\sigma s,\label{equ3}\\
c\sigma \sin_\sigma s+d\cos_\sigma s
&=b\sin_\sigma s+d \cos_\sigma s.\label{equ4}
\end{align}
From equations \eqref{equ1} or, equivalently, \eqref{equ4}, $c\sigma \sin_\sigma s=b\sin_\sigma s,$ or, $b=c\sigma.$ Likewise, from \eqref{equ2} or, equivalently, \eqref{equ3}, 
$d\sin_\sigma s=a\sin_\sigma s$, or, $a=d.$ Therefore, 
$$
B=\begin{pmatrix}
         a&\sigma c\\ c&a
  \end{pmatrix}.
$$
Thus, $B=\lambda A(s_0)$, where $\tan_\sigma^{-1}(s_0)=\frac{c}{a}.$ Because $A(t) $ is a continuous 
one-parameter subgroup of $\GL{\BR},$ $B$ is commuting with every element of $A(t).$ Thus, $B$ belongs to the centralizer of $A(t)$ of $\GL{\BR}.$
\renewcommand{\qedsymbol}{}
\end{proof}
\begin{proof}[Sufficiency]
Demonstration of the sufficiency is straightforward.
\end{proof}
\begin{thm}
Any continuous one-parameter subgroup of $\GL{\BD}$ has, up to similarity and rescaling, the form 
$$
\widetilde{H^\prime_\sigma}(t)=H(t)+\epsilon \lambda H(t+t_0)t,
$$
where $t_0\in \BR,$  $\lambda\in\BR{\setminus}\{0\}$ and $H$ is a subgroup similar to 
$H^\prime_\sigma$ for $\sigma\in\{-1,0,1,r\}$ and where $H^\prime_r=I^\prime.$
\end{thm}
\begin{proof}
Let $B(t)=B_1(t)+\epsilon B_2(t)$ be a continuous one-parameter subgroup in $\GL{\BD}.$ That means that $B(t)=e^{Bt}$ and $B^\prime(0)=B,$ where $B\in \Mat{\BD}.$ Let $B=B_1+\epsilon B_2,$ for some 
$B_1,B_2\in M(\BR).$ Therefore, the continuous one-parameter subgroup 
$$
B(t)=B_1(t)+\epsilon t B_1(t)B_2,
$$
where $B_1(t)$ is a continuous one-parameter subgroup of $\GL{\BR}$ and $B_2$ is a constant matrix in 
$\Mat{\BR}.$ If $B_1(t)$ is a trivial continuous one-parameter subgroup of $\GL{\BR},$ then 
$B(t)=I^\prime.$ If $B_1(t)$ is a non-trivial continuous one-parameter subgroup of of $\GL{\BR},$ then 
$B(t)$ is similar to 
$$
H^\prime_\sigma(t)+\epsilon  H^\prime_\sigma(t)B_2t
$$
where $\sigma\in\{-1,0,1\}.$ For any $s_0,s_1\in\BR,$ 
$$
B(s_0)=H^\prime_\sigma(s_0)+\epsilon H^\prime_\sigma(s_0)B_2 s_0\quad \text{and} \quad
B(s_1)=H^\prime_\sigma(s_1)+\epsilon H^\prime_\sigma(s_1)B_2s_1
$$
are two continuous one-parameter subgroups of $\GL{\BD}.$ Then,
\begin{align*}
B(s_0)B(s_1)
&=H^\prime_\sigma(s_0)H^\prime_\sigma(s_1)+\epsilon(H^\prime_\sigma(s_0)H^\prime_\sigma(s_1)B_2s_1+H^\prime_\sigma(s_0)B_2 s_0H^\prime_\sigma(s_0)),\\ 
B(t+s)
&=H^\prime_\sigma(s_0+s_1)+\epsilon(H^\prime_\sigma(s_0+s_1)(B_2\cdot(s_0+s_1)).
\end{align*}
Because $B(t)$ is a non-trivial continuous one-parameter subgroup of $\GL{\BD},$  
$B(t)B(s)=B(t+s).$ This means that
$$
B_1(t)B_1(s)=B_1(t+s)
$$ 
and
\begin{align*}
H^\prime_\sigma(s_0+s_1)(B_2\cdot(s_0+s_1))
&=H^\prime_\sigma(s_0)H^\prime_\sigma(s_1)(B_2s_1)+H^\prime_\sigma(s_0)(B_2 s_0)H^\prime_\sigma(s_1)),\\
H^\prime_\sigma(s_1)(B_2\cdot(s_0+s_1))
&=H^\prime_\sigma(s_1)(B_2s_1)+(B_2 s_0)H^\prime_\sigma(s_1).
\end{align*}
So, 
$$
H^\prime_\sigma(s_1)B_2s_0=B_2s_0H^\prime_\sigma(s_1).
$$
This means $B_2H^\prime_\sigma(s_1)=H^\prime_\sigma(s_1)B_2.$ Therefore, by Lemma~\ref{leA1}, there are 
$t_0\in\BR$ and $\lambda\in\BR{\setminus}\{0\}$ such that $B_2=\lambda  H^\prime_\sigma(t_0).$ 
Thus, 
\begin{align*}
B_2(s_1)
&=\lambda H^\prime_\sigma(s_1)H^\prime_\sigma(t_0)s_1,\\
B(s)
&=H^\prime_\sigma(s_1)+\epsilon \lambda H^\prime_\sigma(s_1)H^\prime_\sigma(t_0)s_1\in \GL{\BD}.\tag*{\qed}
\end{align*}
\renewcommand{\qed}{} 
\end{proof}
From the preceding, we obtain that, for all $t_0 \in \BR,$ there are the following types of continuous one-parameter subgroups of $\GL{\BD}:$
$$
\widetilde{H^\prime_\sigma}(t)=H^\prime_\sigma(t)+\epsilon\lambda H^\prime_\sigma(t+t_0)t;
$$
where $H^\prime_{\sigma}(t)$ is similar to one of $A^\prime(t),$ $N^\prime(t),$ $K^\prime(t)$ or 
$I^\prime,$ and $\lambda\in\BR{\setminus}\{0\}.$ 
\begin{corollary}
Any non-trivial continuous one-parameter subgroup of $\SL{\BD}$ has the following form:
$$
\widetilde{H_\sigma}(t)=
H_\sigma(t)+\epsilon \lambda te^{\lambda_1 t_0}(H_\sigma(t+t_0)-\cos_\sigma(2t+t_0)H_\sigma(t)),
$$ 
where $t_0\in\BR$ and $\sigma\in\{-1,0,1\}.$
\end{corollary} 
\begin{proof}
Let 
$$
\widetilde{H^\prime_\sigma}(t)=
H^\prime_\sigma(t)+\epsilon\lambda H^\prime_\sigma(t+t_0)t=
e^{\lambda_1 t}H_\sigma(t)+\epsilon\lambda  te^{\lambda_1(t+t_0)} H_\sigma(t+t_0)
$$ 
be a non-trivial continuous one-parameter subgroup of $\GL{\BD}.$ Clearly, 
\begin{align*}
\det(\widetilde{H^\prime\sigma}(t))
&=e^{2\lambda_1 t}+\epsilon \lambda  2te^{\lambda_1(2t+t_0)}\cos_\sigma(2t+t_0),\\ 
\sqrt{\det(\widetilde{H^\prime_\sigma}(t))}
&=\pm e^{\lambda_1 t}\pm\epsilon\lambda te^{\lambda_1(t+t_0)}\cos_\sigma(2t+t_0),\\ 
\frac{1}{\sqrt{\det(\widetilde{H^\prime_\sigma}(t))}}
&=\pm e^{-\lambda_1 t}\mp\epsilon\lambda te^{\lambda_1(t_0-t)}\cos_\sigma(2t+t_0),\\ 
\frac{1}{\sqrt{\det(\widetilde{H^\prime\sigma}(t))}}\widetilde{H^\prime_\sigma}(t)
&=H_\sigma(t)+\\
&\phantom{=}\hspace*{5ex}
\epsilon \lambda te^{\lambda_1 t_0}(H_\sigma(t+t_0)-\cos_\sigma(2t+t_0)H_\sigma(t)).\tag*{\qed}
\end{align*}
\renewcommand{\qedsymbol}{}
\end{proof}
It is evident that we do not lose any interesting types of connected continuous one-parameter subgroups  when we move from $\GL{\BD}$ to $\SL{\BD}.$

Next proposition gives sufficient conditions for the similarity between two continuous one-parameter subgroups in $\GL{\BD}$ and a simple calculation provides a proof.
\begin{proposition}
Let $B_1(t)$ and $\hat{B}_1(t)$ be two continuous one-parameter subgroups of $\GL{\BR}$. Then, 
\begin{itemize}
\item[(a)] $B(t)=B_1(t)+\epsilon\lambda B_1(t+t_0)t$ and 
$\hat{B}(t)=\hat{B}_1(t)+\epsilon \lambda\hat{B}_1(t+t_0)t$ are two continuous one-parameter subgroups of $\GL{\BD}.$ 
\item[(b)] $\hat{B}(t)$ is similar to $B(t)$ if and only if there exists an invertible~$C$ in $\GL{\BR}$ such that $\hat{B}_1(t)=CB_1(t)C^{-1}.$
\end{itemize}
\end{proposition}
\begin{proposition}\label{sldorbit} 
Let $\widetilde{H_\sigma}(t)$ be a continuous one-parameter subgroup of $\SL{\BD}.$ If 
$[u+\epsilon v:1]$ belongs to the $ \widetilde{H_\sigma}(t)$-orbits of $f=[a+\epsilon b:1],$ then: 
\begin{multline*}
v-\lambda e^{\lambda_1t_0}
\tan_\sigma^{-1}\left(\frac{a-u}{au-\sigma}\right)
\left(\frac{(b-a)(u^2-\sigma)}{(a^2-\sigma)}-\right.\\
a^2\left(\frac{(u^2+\sigma)(a^2+\sigma)-4\sigma au}{(a^2-\sigma)^2}\cos_\sigma t_0 + 
    2\sigma\frac{(a-u)(au-\sigma)}{(a^2-\sigma)^2}\sin_\sigma t_0\right)-\\
\hspace*{-36ex}\sigma\left(2\frac{(a-u)(au-\sigma)^3}{(u^2-\sigma)(a^2-\sigma)^3}\cos_\sigma t_0 
 +\right.\\
\hspace*{20ex}\left.\frac{((u^2+\sigma)(a^2+\sigma)-4\sigma au)(au-\sigma)^2}{(u^2-\sigma)(a^2-\sigma)^3}\sin_\sigma t_0\right)\times\\
\hspace*{-18ex}\left.\left(\frac{((u^2+\sigma)(a^2+\sigma)-4\sigma au)(au-\sigma)^2}{(u^2-\sigma)(a^2-\sigma)^3}\cos_\sigma t_0 + \right.\right.\\
\left.\left.2\sigma\frac{(a-u)(au-\sigma)^3}{(u^2-\sigma)(a^2-\sigma)^3}\sin_\sigma t_0\right)\right)=0.
\end{multline*}
\end{proposition}
\begin{proof}
For an arbitrary non-zero point $t\in \BR$ and $t_0\in \BR,$ 
\begin{multline*} 
\widetilde{H_\sigma}(t)f=
\bigg[\frac{a\cos_\sigma t+\sigma\sin_\sigma t}{a\sin_\sigma t+\cos_\sigma t}+\\
\epsilon\lambda t e^{\lambda_1t_0}
\frac{b-a-a^2\cos _\sigma(2t+t_0)\sin_\sigma t_0-\frac{\sigma}{2}\sin_\sigma(4t+2t_0))}{(a\sin_\sigma t+\cos_\sigma t)^2}:1\bigg].
\end{multline*}  
Let us define 
$$
u=\frac{a\cos_\sigma t+\sigma\sin_\sigma t}{a\sin_\sigma t+\cos_\sigma t}=\frac{a+\sigma\tan_\sigma t}{a\tan_\sigma t+1}.
$$
A simple calculation gives $\tan_\sigma t=\frac{a-u}{au-\sigma}.$ Then,
%%%New as of 9/17/2018
\begin{align*}
v
&=\lambda t e^{\lambda_1t_0}\frac{b-a-a^2\cos _\sigma(2t+t_0)\sin_\sigma t_0-\frac{\sigma}{2}\sin_\sigma(2(2t+t_0))}{(a\sin_\sigma t+\cos_\sigma t)^2}\\
&=\frac{\lambda t e^{\lambda_1t_0}}{(a\sin_\sigma t+\cos_\sigma t)^2}\left(b-a-\right.\\
&\hspace*{5ex} a^2(((\cos^2_\sigma t+\sigma\sin^2_\sigma t)\cos_\sigma t_0+2\sigma\sin_\sigma t\cos_\sigma t)\sin_\sigma t_0)-\\
&\hspace*{5ex}\sigma(2\sin_\sigma t\cos_\sigma t\cos_\sigma t_0 +(\cos^2 _\sigma t+\sigma\sin^2_\sigma t)\sin_\sigma t_0)\times\\
&\hspace*{7ex}((\cos^2 _\sigma t+\sigma\sin^2_\sigma t)\cos_\sigma t_0 +2\sigma \left.\sin_\sigma t\cos_\sigma t\sin_\sigma t_0)\right)\\
&=\frac{\lambda t e^{\lambda_1t_0}}{(1-\sigma\tan^2_\sigma t)(a\tan_\sigma t+1)^2}
((b-a)(1-\sigma\tan^2_\sigma t)^2-\\
&\hspace*{5ex}
a^2((1+\sigma\tan^2_\sigma t)(1-\sigma\tan^2 _\sigma t) \cos_\sigma t_0+\\
&\hspace*{25ex}
2\sigma\tan_\sigma t(1-\sigma\tan_\sigma^2 t))\sin_\sigma t_0-\\
&\hspace*{5ex}
\sigma(2\tan_\sigma t\cos_\sigma t_0+(1+\sigma\tan^2_\sigma t)\sin_\sigma t_0)\times\\
&\hspace*{25ex}
((1+\sigma\tan^2_\sigma t)\cos_\sigma t_0 +2\sigma \tan_\sigma t\sin_\sigma t_0)).
\end{align*}
After substituting $\tan_\sigma t$ for $v,$ one obtains the following result: 
%%% New as of 9/17/2018
\begin{align*} 
v=
&\frac{\lambda\tan_\sigma^{-1} \frac{a-u}{au-\sigma} e^{\lambda_1t_0}(au-\sigma)^4}{(u^2-\sigma)(a^2-\sigma)^3}\left((b-a)\frac{(u^2-\sigma)^2(a^2-\sigma)^2}{(au-\sigma)^4}-\right.\\[1ex]
&a^2\left(\frac{((u^2+\sigma)(a^2+\sigma)-4\sigma au)(u^2-\sigma)(a^2-\sigma)}{(au-\sigma)^4}\cos_\sigma t_0\right.\\[1ex]
& +\left.2\sigma\frac{(a-u)(u^2-\sigma)(a^2-\sigma)}{(au-\sigma)^3}\right)\sin_\sigma   t_0-\sigma\left(2\frac{a-u}{au-\sigma}\cos_\sigma t_0\right.\\[1ex]
& +\left.\frac{(u^2+\sigma)(a^2+\sigma)-4\sigma au}{(au-\sigma)^2}\sin_\sigma t_0\right) 
   \left(\frac{(u^2+\sigma)(a^2+\sigma)-4\sigma au}{(au-\sigma)^2}\cos_\sigma t_0\right.\\[1ex]
& \hspace*{45ex}
\left.\left.+2\sigma\frac{a-u}{au-\sigma}\sin_\sigma t_0 \right)\right)\\[1ex]
&=\lambda e^{\lambda_1t_0}\tan_\sigma^{-1} \left(\frac{a-u}{au-\sigma}\right) \times\\[1ex]
& \hspace*{5ex}\left(\frac{(b-a)(u^2-\sigma)}{(a^2-\sigma)}-a^2\left(\frac{(u^2+\sigma)(a^2+\sigma)-4\sigma au}{(a^2-\sigma)^2}\right.\right.\\[1ex]
&\left. \cos_\sigma t_0 +2\sigma\frac{(a-u)(au-\sigma)}{(a^2-\sigma)^2}\right)\sin_\sigma t_0-\sigma\left(2\frac{(a-u)(au-\sigma)^3}{(u^2-\sigma)(a^2-\sigma)^3}\cos_\sigma t_0 \right.\\[1ex]
&\left.+\frac{((u^2+\sigma)(a^2+\sigma)-4\sigma au)(au-\sigma)^2}{(u^2-\sigma)(a^2-\sigma)^3}\sin_\sigma t_0\right) \times\\[1ex]
& \left(\frac{((u^2+\sigma)(a^2+\sigma)-4\sigma au)(au-\sigma)^2}{(u^2-\sigma)(a^2-\sigma)^3}\cos_\sigma t_0 +\right.\\
&\hspace*{28ex}\left.\left.2\sigma\frac{(a-u)(au-\sigma)^3}{(u^2-\sigma)(a^2-\sigma)^3}\sin_\sigma t_0\right)\right).
\end{align*} 
\renewcommand{\qedsymbol}{}
\end{proof}
\section*{Acknowledgments}
I am grateful to the Iraqi government for its support in a form of a scholarship. I would like to thank Prof. Vladimir Kisil for valuable discussions and important remarks. I am also grateful to the anonymous referees for many useful suggestions which were applied to improve the paper.
\bibliographystyle{plain}
%\bibliography{ref13}{}
\bibliography{ref13rev}{}
\end{document}